\documentclass{amsart}
\usepackage{amsthm}

\newtheorem{theorem}{Theorem}[section]
\newtheorem*{theorem_T_4.2_BG5}{Theorem B} 
\newtheorem*{theoremTensorProdDelta}{Theorem A}
\newtheorem*{theoremn_hol_sim_mod}{Theorem C}
\newtheorem{lemma}[theorem]{Lemma}

\newtheorem{proposition}[theorem]{Proposition}
\newtheorem{definition}[theorem]{Definition}

\newtheorem{example}[theorem]{Example}
\DeclareMathOperator{\rk}{rank}
\DeclareMathOperator{\gk}{GK-dim}
\DeclareMathOperator{\ann}{ann}

\DeclareMathOperator{\s}{\ast}

\DeclareMathOperator{\supp}{Supp}
\DeclareMathOperator{\im}{Image}
\DeclareMathOperator{\hm}{Hom}

\begin{document}
\title{Modules over quantum Laurent polynomials}

\author{Ashish Gupta}

\begin{abstract}   
It is shown that the Gelfand--Kirillov dimension 
for modules over quantum Laurent polynomials is additive 
with respect to tensor products over the base field. The 
Brookes--Groves invariant associated with a tensor product of modules 
is determined. Strongly holonomic modules are studied and
it is shown that there can be nonholonomic simple modules.   
\end{abstract}

\maketitle

\section{Introduction}
\label{introd}
Let $F$ be a field. 
Consider the associative $F$-algebra $P(\mathfrak q)$ generated by 
$u_1, \cdots, u_n$ and their inverses such that 
\begin{equation}
\label{def_rel_2}
u_iu_j = q_{ij}u_ju_i \ \ \ \forall i,j \in \{ 1,\cdots, n \},  
\end{equation}
where $q_{ij}$ are nonzero scalars in $F$ and $\mathfrak q = (q_{ij})$.
It is known by various names such as the
\emph{multiplicative analogue of the Weyl algebra},
the \emph{quantum Laurent polynomial algebra} and the \emph{quantum torus}.
It has the structure of a twisted group algebra $F \s A$ of a free abelian group $A$ of rank $n$ over $F$.   

In the special case $n = 2$, (\ref{def_rel_2}) becomes 
$u_1u_2 = q_{12}u_2u_1$, where $q_{12} \in F - \{0\}$.
This situation was first studied in \cite{Ja} and 
\cite{Lo} and it was shown that when $q_{12}$ is not a root of unity in $F$,
$P((q_{12}))$ shares certain curious properties with the first Weyl algebra $A_1(k)$ 
over a field $k$ of characteristic zero.
 
The case of arbitrary $n$ was first considered by J.~C.~McConnell and J.~J.~Pettit in \cite{MP}.
It was shown in \cite{MP} that if the subgroup of the multiplicative group of 
$F$ generated by the $q_{ij}$ has the maximal possible torsion-free rank then  
$P(\mathfrak q)$ is a simple noetherian hereditary domain.
 
The quantum Laurent polynomial algebras play a fundamental role in noncommutative geometry (see \cite{M}). 
They also arise in the representation theory of torsion-free nilpotent groups as suitable localizations (see \cite{Br2}).
   
In recent times there has been considerable interest in the theory of these algebras and their generalizations.
The ring-theoretic properties of these rings have been studied in \cite{AG}, \cite{A4} and \cite{MP}. 
In \cite{A1}, \cite{A2} and \cite{A3}, projective and simple modules over general quantum polynomial rings have  been considered by V.~A.~Artamonov.
In \cite{BG1} and \cite{BG2}, C.~J.~B.~Brookes and J.~R.~J.~Groves have introduced a geometric invariant for 
$F \s A$-modules modelled 
on the original Bieri--Strebel invariant (see 
\cite{BS}).     

The algebras $P(\mathfrak q)$ are precisely the twisted group algebras $F \s A$ of a free finitely generated abelian group $A$ over $F$.
In this paper we consider the structure of modules over the algebras $F \s A$.
We first review (Section 1.1) the basic properties of these algebras.   
We then give a brief exposition of the 
geometric invariant $\Delta(M)$ of Brookes and 
Groves associated with a finitely generated $F \s A$-module $M$.
Theorem A of Section 3 determines the
Brookes--Groves invariant associated with a tensor
product of modules. The Gelfand--Kirillov dimension (GK dimension)
of a tensor product of modules is also determined.

\begin{theoremTensorProdDelta}
Let $M_i$ be a finitely generated $F \ast
A_i$-module, where $i  = 1,2$. Then for the finitely
generated $F \s (A_1 \oplus A_2)$-module $M_1 \otimes_F M_2$,
\[ \Delta(M_1 \otimes_F M_2) = p_1^*\Delta(M_1) +
p_2^*\Delta(M_2),
\]
where for $i = 1, 2$, 
$p_i^* : A_i^* \rightarrow (A_1 \oplus A_2)^*$ is the injection induced by the projection $p_i : A_1 \oplus A_2 \rightarrow A_i$. Furthermore,
\[ \gk(M_1 \otimes_F M_2) = \gk(M_1) + \gk(M_2). \]
\end{theoremTensorProdDelta}

Section 4 is concerned with strongly holonomic modules. 
These are defined analogously to the holonomic $A_n(k)$-modules, where $A_n(k)$ denotes the $n$-th Weyl algebra over a field $k$ of characteristic zero. An $A_n$-module $N$ is called \emph{holonomic} if $\gk(N) = \frac{1}{2}\gk(A_n)$.
Holonomic $A_n$-modules form an important subclass of $A_n$-modules and possess some nice properties (see \cite{Bj}). For the algebras $F \s A$, $\gk(F \s A) = \rk(A)$ by \cite[Section 5.1]{MP}.   
We may thus call an $F \s A$-module $M$ holonomic if $\gk(M) = \frac{1}{2}\rk(A)$. Such modules are encountered in group theory (see \cite{BG5}) with the additional condition that $M$ is torsion-free as $F \s B$-module whenever
$B$ is a subgroup of $A$ with $F\s B$ commutative. The following theorem first shown in \cite{BG5} is given a new proof in Section \ref{s_h_m}.
   
\begin{theorem_T_4.2_BG5}
Suppose that an algebra $F \s A$ with center $F$ and $\rk(A) = 2m$ has a strongly holonomic module.
Then for a finite index subgroup $A'$ in $A$,
\[ F \s A' =  F \s B_1 \otimes_F  \cdots \otimes_F F \s B_m, \]
where each $B_i \cong \mathbb Z \oplus \mathbb Z$ and
$m = \frac{1}{2}\rk(A)$.
\end{theorem_T_4.2_BG5}

In \cite[Section 6]{MP}, the question whether an algebra $F \s A$ that is simple 
 can have simple modules with distinct GK dimensions was considered. 
It was shown that if  $F \s A$ has Krull (global) dimension one then each simple 
$F \s A$-module has GK dimension equal to $\rk(A) -1$.  
In fact if an algebra $F \s A$ has dimension $m$, 
where $1 \le m \le \rk(A)$, the work of Brookes in \cite{Br2} 
implies that the minimum possible GK dimension for a nonzero finitely generated $F \s A$-module is 
$\rk(A) - m$. 
The question then arises 
if the GK dimension of a simple $F \s A$-module always
equals this minimum as in the dimension one case. 
 
We show in Section 5 that this need not be true in general.
 
\begin{theoremn_hol_sim_mod}
\label{n_hol_sim_mod}
Suppose that $F \s A$
has center exactly $F$
and $A$ has
a subgroup $B$
with $A/B$ infinite cyclic such that
$F \s B$ is commutative.
Then $F \s A$ has a simple $F \s B$-torsion-free module $S$ with 
$\gk(S) = n - 1$.
\end{theoremn_hol_sim_mod}

\subsection{Basic properties}
\label{twst_grp_alg}
We will now discuss the basic properties of the algebra $P(\mathfrak q)$ and its modules.
It is easily seen that the monomials $u_1^{m_1} \cdots u_n^{m_n}$, where
$m_j \in \mathbb Z$, 
constitute an $F$-basis of $P(\mathfrak q)$. The monomial $u_1^{m_1} \cdots u_n^{m_n}$ is denoted by $\mathbf {u}^{\mathbf {m} }$, where $\mathbf {m} = (m_1, \cdots, m_n) \in \mathbb Z^n$. We denote the set of nonzero elements of $F$ by $F^*$. The facts in the next proposition were established in \cite[Section 1]{MP}. 

\begin{proposition}[McConnell and Pettit]
\label{el_prop}
The following properties hold for the algebra $P(\mathfrak q)$:

\begin{itemize}
\item[(i)] 
$\mathbf {u}^{\mathbf{m}}\mathbf {u}^{\mathbf {m'}} =
\prod_{j > i} q_{ji}^{m_jm_i'} \mathbf { u}^{\mathbf { m + m'}}$,

\item[(ii)]  $(\mathbf { u}^{\mathbf { m}})^{-1} =
\mu(\mathbf { m})\mathbf { u}^{- \mathbf { m}}$, where
$\mu(\mathbf { m}) = \prod_{j > i}q_{ji}^{m_jm_i}$,

\item[(iii)] if $\alpha \in P(\mathfrak q)$ then $\alpha$ is a unit if and only if
$\alpha = \lambda \mathbf { u}^{\mathbf { m}}$ for some nonzero $\lambda \in F$,

\item[(iv)]  the group-theoretic commutator
$[\mathbf { u}^{\mathbf { a}}, \mathbf { u}^{\mathbf { b}}] =
\mathbf { u}^{\mathbf { a}}\mathbf { u}^{\mathbf { b}}
(\mathbf { u}^{\mathbf { a}})^{-1}
(\mathbf { u}^{\mathbf { b}})^{-1}$
lies in $F^*$,

\item[(v)] the derived subgroup of the group of units of $P(\mathfrak q)$ coincides with the subgroup of $F^*$ generated by the $q_{ij}$, where $1 \le i, j \le n$,

\item[(vii)] $P(\mathfrak q)$ is simple if and only if it has center exactly $F$.
\end{itemize}
\end{proposition}

An associative $F$-algebra $\mathcal A$ is
a twisted group algebra $F \s A$ of a finitely generated
free abelian group $A$ over the field $F$ if 
\begin{itemize}
\item[(i)] there is an injective function
$\bar{\ }: A \rightarrow \mathcal A$,
$a \mapsto \bar a$, such that
$\overline A: = \im(\bar{\ })$ is a basis of $\mathcal A$ as an $F$-space,
\item[(ii)] the multiplication in $\mathcal A$ satisfies
\begin{equation}
\label{tw}
 \bar {a}_1 \bar {a}_2 = \tau(a_1, a_2)\overline {a_1a_2} \ \ \ \forall a_1, a_2 \in A,  
\end{equation}
where $\tau: A \times A \rightarrow F^*$ 
is a function satisfying: 
\begin{equation}
\label{cc}
 \tau(a_1, a_2)\tau(a_1a_2, a_3) = \tau(a_2, a_3)\tau(a_1, a_2a_3) \ \ \ \forall a_1, a_2, a_3 \in A. 
\end{equation}
\end{itemize}

Let $A$ be a free abelian group with basis $\{a_1, \cdots, a_n\}$.
Then there is an injection $A \rightarrow P(\mathfrak q)$ defined by  
$\prod a_i^{m_i} \mapsto \prod u_1^{m_i}$, where $m_i \in \mathbb Z$ and $i = 1, \cdots, n$. 
Condition (ii) above easily follows from (\ref{def_rel_2}). Finally, the associativity of 
$P(\mathfrak q)$ implies (\ref{cc}). 
Hence $P(\mathfrak q)$ is a twisted group algebra $F \s A$.

We note that in an algebra $F \s A$, the 
scalars are central so that 
\[ \lambda \bar a = \bar a \lambda \ \ \ \forall \lambda \in F, a \in A.\]   
In a \emph{crossed product}
 $D \s A$ (see \cite[Chapter 1]{Pa2}), where $D$ is a division ring, the multiplication is defined as in (ii) above but an element $d \in D$ need not be central. In fact for all $a \in A$ and $d \in D$,
\[ \bar ad = \sigma_a(d)\bar a,    \] 
for an automorphism $\sigma_a$ of $D$. 

Given an algebra $F \s A$, we may express $\alpha \in F \s A$ 
uniquely as
$\alpha  = \sum_{a \in A}\kappa_a \bar a$, where $\kappa_a \in F$.
The subset
$\supp(\alpha) := \{a \in A \mid \kappa_a \ne 0 \}$
of $A$ is finite and is called
the 
\emph{support of $\alpha$ in $A$}.
   For a subgroup $B$ of $A$ the subalgebra $\{\beta \in F \s A \mid \supp(\beta) \subseteq B \}$ of $F \s A$ is a twisted group algebra $F \s B$ of $B$ over $F$.

It is known (see, for example, \cite[Lemma 37.8]{Pa2}) that if $B$ is subgroup of $A$ then
$S_B = F \s B \setminus \{0\}$ is an Ore subset in $F \s A$.
As a consequence the subset 
\[ T_{S_B}(M) := \{ x \in M \mid x.s = 0 \  \mbox{for some} \ s \in S_B \} \] 
of $M$ is an $F \s A$-submodule of 
$M$. 
We say that $M$ is 
$S_B$-torsion (or $F \s B$-torsion) if $T_{S_B}(M) = M$ and $S_B$-torsion-free ($F \s B$-torsion-free)
if $T_{S_B}(M) = 0$. We note that the right Ore localization $(F \s A)S_B^{-1}$ is a crossed product    
$D_B \s A/B$, where $D_B$ stands for the quotient division ring of $F \s B$. We shall also write $(F \s A)(F \s B)^{-1}$ for $(F \s A)S_B^{-1}$.  

Note that if $a \in A$ then $\bar a$ is a unit of $F \s A$. Without loss, we may assume that $\bar 1$ is the identity of $F \s A$.
It easily follows from (\ref{tw}) that for $a_1, a_2 \in A$, the group-theoretic commutator 
$[\overline{ a_1}, \overline {a_2}] = \bar{a}_1\bar{a}_2\bar{a}_1^{-1}\bar{a}_2^{-1} \in F$. 
Then the following equalities hold (see \cite[Section 5.1.5]{Ro}):

\begin{align}
\label{c_m_r_1}
[\bar {a}_1 \bar {a}_2, \bar {a}_3] &= [\bar {a}_1, \bar {a}_3] [\bar {a}_2, \bar {a}_3],\\ 
\label{c_m_r_2}
[\bar {a}_1,  \bar {a}_2 \bar {a}_3] &= [\bar {a}_1, \bar {a}_2] [\bar {a}_1, \bar {a}_3]. 
\end{align}

For a subset $X \subseteq A$, we define $\overline X = \{ \bar x \mid x \in X \}$. 
Moreover, if $X_1, X_2 \subseteq A$, we define 
$[\overline{X_1}, \overline{X_2}] = \langle [\bar{x}_1, \bar{x}_2] \mid x_1 \in X_1, x_2 \in X_2  \rangle$. 
It is clear that $[\overline{X_1}, \overline{X_2}]$ is a subgroup of the multiplicative group $F^*$. 
 
\section{The Brookes--Groves geometric invariant}
\label{BGinv}

We shall now describe a geometric invariant which was introduced in \cite{BG1} and \cite{BG2}.
It is defined for finitely generated modules over a crossed product $D \s A$ of a finitely generated free abelian group $A$ over a division ring $D$. Since a twisted group algebra $F \s A$ is a special case of $D \s A$, the definitions and 
theorems that follow apply to $F \s A$-modules as well.   

Let $A$ be a 
finitely generated free 
abelian group 
and $A^* := \hm_{\mathbb Z}(A, \mathbb R)$.  Then $A^*$ is
an $\mathbb R$-space with
$\dim(A^*) = \rk(A)$, where $\rk(A)$ is the cardinality of a
basis of $A$.
For a basis $\mathfrak b = \{ b_i \mid i \in I \}$ of $A$
we recall that there is a basis $\mathfrak b^* =
\{ b^*_i \mid i \in I \}$ dual to $\mathfrak b$
and this allows the
construction of an isomorphism $\mathbb R^{| \mathfrak b |}
\rightarrow  A^*$. We may thus speak of characters $\phi \in A^*$ as points. 
There is a $\mathbb Z$-bilinear map $\langle - , - \rangle : A^* \times A \rightarrow \mathbb R$ defined by
\[  (\phi, c) \mapsto  \langle \phi, c \rangle = \phi(c) \ \ \  \forall \phi \in A^*, c \in A. \]
Whenever $B \le A$ is a subgroup, \[ \ann(B) :=
\{\phi \in A^* \mid \langle \phi, B \rangle = 0\} \] is a subspace of $A^*$
with \[ \dim(\ann(B)) = \rk(A) -  \rk(B). \]
For a subspace $V \le A^*$,  we define $\ann(V)$
analogously as \[ \ann(V) = \{ b \in A \mid \langle V, b \rangle
= 0\}. \]  It is not difficult to show that $\ann ( \ann(B) ) = B$.
For a point $\phi \in A^*$,  we define
\begin{align*}
A_{\phi, 0}  &= \{ a \in A \mid  \phi(a) \ge 0 \}, \\
A_{\phi, +}  &= \{ a \in A \mid  \phi(a) > 0 \}.
\end{align*}

Note that 
$A_{\phi, 0}$ is a submonoid and
$A_{\phi, +}$ a subsemigroup of $A$. 
In \cite[Proposition 3.1]{BG2} several equivalent
definitions of the geometric invariant are given which are analogous 
to the commutative case (see \cite{BS}).
The following definition was used in \cite{Br2}.

\begin{definition}[Brookes and Groves]
\label{def_BG_invnt}
Let $D$ be a division ring and $A$ be a free finitely
generated abelian group. Let $M$ be a finitely generated $D
\ast A$-module with a finite generating set $\mathcal X$.
Then $\Delta(M)$ is defined as the subset \[ \Delta(M) = \{
\theta \in A^* \mid  \mathcal XA_{\theta, 0} >
\mathcal XA_{\theta, +} \} \] of
$A^*$.
\end{definition}
The above definition seems to depend 
on the
choice of a generating set $\mathcal X$
for $M$ but 
the subset $\Delta(M)$ so
defined is actually independent
of such a choice (see 
\cite[Section 2]{Br2}). 

\begin{definition}[Definition 2.1 of \cite{BG3}]
Let $M$ be a finitely generated $D \s A$-module. 
For a point  $\phi \in A^*$, the \emph{trailing coefficient module} $TC_\phi(M)$ of $M$ at $\phi$ is defined as $TC_\phi(M)  = \mathcal X A_{\phi, 0}/ \mathcal X A_{\phi, +}$, where $\mathcal X$ is a (finite) generating set for $M$. 
\end{definition}
Note that $TC_\phi(M)$ is a finitely generated $D \s K$-module where $K = \ker \phi$. It is immediate from Definition \ref{def_BG_invnt} that $\phi \in \Delta(M)$ if and only if $TC_\phi(M) \ne 0$. In general, $TC_\phi(M)$ need not be independent of $\mathcal X$. A dimension for finitely generated $D \s A$-modules was introduced in \cite{BG2}. 

\begin{definition}[Definition 2.1 of \cite{BG2}]
\label{def_dim}
Let $M$ be a $D \s A$-module. The dimension $\dim (M)$ of
$M$
is defined to be
the greatest integer $r$, where $0 \le r \le \rk(A)$, so that
for some subgroup $B$ in
$A$ with rank $r$, 
$M$ is not $D \s B$-torsion.
\end{definition}

It was shown in \cite{BG2} that $\dim(M)$ coincides with
the Gelfand--Kirillov dimension (GK dimension) of $M$. We 
shall thus occasionally write 
$\gk(M)$ for $\dim(M)$. 
The following useful fact was also shown in 
\cite{BG2}.  
\begin{proposition}[Lemma 2.2 of \cite{BG2}]
\label{dim_is_exact}
Let \[
0 \rightarrow M_1 \rightarrow M \rightarrow M_2 \rightarrow 0 \]
be an exact sequence of $D \s A$-modules. Then
\[ \dim(M) = \sup\{\dim(M_1), \dim(M_2) \}. \]
\end{proposition}

As already noted we may
identify $A^*$ with
$\mathbb R^n$ and $\Delta(M) \subseteq A^*$ is thus identified
with a subset of $\mathbb R^n$.
A subset $S$ of $\mathbb R^n$
is a \emph{polyhedron} when $S$ is a finite union of
\emph{convex polyhedra}. A \emph{convex polyhedron} is an
intersection of  finitely many closed half spaces in
$\mathbb R^n$. A polyhedron is \emph{rational} when each of
the boundaries
of the half spaces used to define it is rational, that is,
when it is generated by  rational linear combinations of the
chosen dual basis. For a convex polyhedron $\mathcal C$,
the
dimension of $\mathcal C$ is the dimension
of the subspace of
$\mathbb R^n$ spanned by $\mathcal C$.
The dimension of
a polyhedron is the greatest of the dimensions of its
constituent convex polyhedra.
In \cite[Theorem 4.4]{BG2}, it was shown that 
an ``essential" subset of  
$\Delta(M)$ is a polyhedron of dimension equal to the GK dimension of $M$.
It was shown in \cite{Wa1} that the Brookes--Groves 
invariant is polyhedral. 

\begin{theorem}[Theorem A of \cite{Wa1}]
\label{BG2_Theorem_4.4}
If $D \s A$ is a crossed 
product of a division ring 
$D$ by a free finitely 
generated  abelian
group $A$, then for all 
finitely generated $D \s A$-modules $M$, 
$\Delta(M)$ is a closed 
rational polyhedral 
cone in 
$\hm_{\mathbb Z}(A, \mathbb R )$.
\end{theorem}

The next section gives an application of the geometric
invariant to tensor products of $F \s A$-modules.

\section{The geometric invariant and
tensor products}
\label{nc_add_th}

Given twisted group algebras 
$F \s A_1$ and
$F \s A_2$, the tensor product
\[ F \s A_1 \otimes_F F \s A_2 \]
of $F$-algebras is a twisted 
group algebra of $A_1 \oplus A_2$ over $F$. 
Moreover, if $M_1$ and $M_2$
are modules over $F \s A_1$ and
$F \s A_2$ respectively then $M_1 \otimes_F M_2$ becomes
an $F \s A_1 \otimes_F F \s A_2$-module via
\[ (m_1 \otimes m_2)(\bar a_1, \bar a_2)
= m_1\bar a_1 \otimes m_2\bar a_2  \ \ \ \forall m_1,m_2 \in M, a_1,a_2 \in A \] 

We shall now determine 
the Brookes--Groves invariant 
associated with such a tensor 
product of modules. 

\begin{theoremTensorProdDelta}
Let $M_i$ be a finitely generated module over $F \ast
A_i$, where $i  = 1,2$. Then for the finitely
generated $F \s (A_1 \oplus A_2)$-module $M_1 \otimes_F M_2$,

\begin{equation} 
\label{tsr_prod_delta}
\Delta(M_1 \otimes_F M_2) = p_1^*\Delta(M_1) +
p_2^*\Delta(M_2),
\end{equation}
where $p_i^* : A_i^* \rightarrow (A_1 \oplus A_2)^*$ is the injection induced by the projection $p_i : A_1 \oplus A_2 \rightarrow A_i$ for $i \in \{1,2\}$. Furthermore,
\begin{equation}
\label{tsr_GK_min}
 \gk(M_1 \otimes_F M_2) = \gk(M_1) + \gk(M_2).
\end{equation}
\end{theoremTensorProdDelta}
\begin{proof}
Let $M : = M_1
\otimes_F M_2$.
We shall first show that
\begin{equation}
\label{lsinrs}
p_1^*\Delta(M_1) + 
p_2^*\Delta(M_2) \subseteq \Delta(M).
\end{equation}

We shall utilize \cite[Section 3, Definition 4]{BG2} for the
$\Delta$-set of a module.
This is as follows: for a finitely generated $F \s
A$-module
$L$ and a point $\phi \in A^*$
a nontrivial $\phi$-filtration of $L$ is a family of
$F$-subspaces $L_\mu$ of $L$, where $\mu \in \mathbb R$,
such that:
\begin{enumerate}
\item[(C1)] $L_\nu \ge L_\mu$, whenever $\nu \le \mu$,
\item[(C2)] $\cup_{\mu \in \mathbb R} L_{\mu} = L$,
\item[(C3)]  $L_\mu \bar a = L_{\mu +
    \phi(a)}$ for any $a \in A$,
\item[(C4)] for each $\mu \in \mathbb R$,
the subspace $L_\mu$ is a proper subspace of $L$.
\end{enumerate}
Then $\Delta(M)$ is defined to be the set of all
$\phi \in A^*$
for which
there exists a nontrivial $\phi$-filtration together with the zero of $A^*$. This definition is equivalent to Definition \ref{def_BG_invnt} (see \cite[Proposition 3.1]{BG2}). 

Thus to show (\ref{lsinrs}) it suffices to show that for
$\phi_i \in \Delta(M_i)$ such that either $\phi_1$ or $\phi_2$ is nonzero,
$M$ has a nontrivial $\phi: = \phi_1p_1 +
\phi_2p_2$-filtration.
Suppose, for the moment, that $\phi_i \ne 0$ for $i = 1,2$.
Since $\phi_i \in \Delta(M_i)$, where $1 \le i \le 2$, there
exists a nontrivial $\phi_i$-filtration
$\{M_{i}^{\mu}\}_{\mu \in \mathbb R}$ of $M_i$. 
We now define a $\phi$-filtration on $M$ by setting 
\[ M_\lambda = \sum_{\mu + \nu =
\lambda} M_{1}^\mu \otimes_F M_{2}^\nu, \ \ \ \forall \lambda 
\in
\mathbb R \ \mbox{and} \ \forall (\mu, \nu) \in \{ \mathbb R^2 \mid \mu + \nu = \lambda \}  \]  
and verify the above conditions (C1) -- (C4) as follows:

\begin{itemize}
\item[(C1)]
If $\lambda_1 \le \lambda_2$ are real numbers and
$(\mu_2, \nu_2)$ is any real pair such that $\lambda_2
=
\mu_2 + \nu_2$, then we can find a real pair $(\mu_1,
\nu_1)$ such that $\lambda_1= \mu_1 + \nu_1$ and such
that $\mu_1 \le \mu_2$, $\nu_1 \le \nu_2$. But then
$M_1^{\mu_1} \ge M_1^{\mu_2}$ and $M_2^{\nu_1} \ge
M_2^{\nu_2}$, whence \[ M_1^{\mu_1} \otimes_F
M_2^{\nu_1} \ge M_1^{\mu_2} \otimes_F M_2^{\nu_2}, \]
which shows that $M_{\lambda_1} \ge M_{\lambda_2}$. Hence (C1) holds. 

\item[(C2)]
As the elements of $M$ may be expressed as finite sums
of the decomposable elements
$x_1 \otimes x_2$, where $x_i \in M_i$, hence to
see that $M = \bigcup_{\lambda \in \mathbb R}
M_\lambda$ it is sufficient to show that $x_1 \otimes
x_2 \in M_\lambda$ for some $\lambda \in \mathbb R$.
But the filtrations $\{M_1^\mu\}$ and
$\{M_2^\nu\}$ guarantee the existence of
real numbers $\mu$ and $\nu$
such that $x_1 \in M_1^\mu$ and $x_2 \in M_2^\nu$. But then $x_1 \otimes x_2 \in M_1^\mu
\otimes_F M_2^\nu \subseteq M_{\mu + \nu}$.

\item[(C3)] To show (C3) we note that
\begin{align*}
M_\lambda(\overline{a_1}, \overline{a_2}) &=
\Bigl ( \sum_{\mu + \nu
= \lambda} M_1^\mu \otimes_F M_2^\nu \Bigr
)(\overline{a_1}, \overline{a_2}) \\
&= \sum_{\mu + \nu = \lambda} M_1^\mu\overline{a_
1} \otimes_F
M_2^\nu\overline{a_2} \\ &= \sum_{\mu + \nu = \lambda}
M_1^{\mu + \phi_1(a_1)} \otimes_F M_2^{\nu +
\phi_2(a_2)} \\ &=
\sum_{\mu' + \nu' = \lambda +
\phi((a_1,a_2))} M_1^{\mu'} \otimes_F M_2^{\nu'} \\ &=
M_{\lambda + \phi((a_1, a_2))}.
\end{align*}

\item[(C4)] To show (C4) we suppose to the contrary that for some $\lambda
    \in \mathbb R$ we have $M_\lambda = M$. This is equivalent to asserting that $M_0 = M$. We shall show that this 
results in a contradiction.
By (C4) for the nontrivial $\phi_i$ filtration on $M_i$, where $i = 1, 2$, $M_i/M_i^0$ 
is a nonzero $F$-vector space. We fix an $F$-basis
$\mathcal B_i^0$ of $M_i^0$, and an $F$-basis $\mathcal
B_i$ of $M_i$ such that $\mathcal B_i^0 \subseteq
\mathcal
B_i$. Note that the inclusion $\mathcal B_i^0 \subseteq
\mathcal B_i$ must be strict since $M_i/M_i^0$ is
nonzero. Pick $u_i \in \mathcal B_i \setminus \mathcal
B_i^0$.
Now $\mathcal B_1 \otimes \mathcal B_2: = \{
v_1 \otimes v_2 \mid v_i \in B_i \}$ is an $F$-basis for
$M = M_1 \otimes_F M_2$ .
Moreover, the element $(u_1 \otimes u_2)$ of $\mathcal
B: =  \mathcal B_1 \otimes \mathcal B_2$ does not lie in the
subset \[ \mathcal B' : = (\mathcal B_1^0 \otimes
\mathcal
B_2) \cup (\mathcal B_1 \otimes \mathcal B_2^0), \] 
where \[ \mathcal B_1^0 \otimes
\mathcal 
B_2 := \{w 
\otimes v \mid
w \in \mathcal B_1^0, v \in \mathcal B_2\} \] and $\mathcal B_1 \otimes
\mathcal B_2^0$ is defined analogously. Since
$\mathcal B$ is a basis of $M$, $u_1 \otimes u_2$
is not contained in \[ M' : = M_1^0 \otimes_F M_2 + M_1
\otimes_F
M_2^0 \] 
which is the $F$-linear span of $\mathcal B'$, and
\emph{a fortiori}, $u_1 \otimes u_2$ is not in
\[ M_0 = \sum_{\mu \in \mathbb R} M_1^\mu \otimes M_2^{- \mu} = \sum_{\nu \ge 0} M_1^\nu \otimes M_2^{- \nu} + \sum_{\nu \le 0} M_1^\nu \otimes M_2^{- \nu}. \] 
Hence, $M_0 \ne M$.
\end{itemize}

We have thus exhibited a nontrivial $\phi$-filtration of $M$ and so $\phi \in \Delta(M)$. It follows that if 
$\phi \in \Delta(M_i) \setminus \{0\}$ then 
$\phi = \sum_{i = 1}^2 \phi_ip_i$ is in $\Delta(M)$.
The case when either $\phi_1$ or $\phi_2$ is zero is handled similarly. 

We now show the reverse inclusion of (\ref{lsinrs}).
Let $\psi \in \Delta(M)$.
For $i = 1,2$, we define $\psi_i \in A_i^*$ by
$\psi_i := \psi
e_i$,where
$e_i : A_i \rightarrow A_1 \oplus A_2$ is the
injection of the biproduct. We shall show that
$\psi_i \in \Delta(M_i)$. It then follows that
\[ \psi = \psi_1p_1 + \psi_2p_2 \in \sum_{i = 1}^2 \Delta(M_i)p_i. \]

Suppose that $\psi_1 \not \in \Delta(M_1)$.
Let $\mathcal X_1$
be a finite generating set for $M_1$.
By \cite[Proposition 3.1(v)]{BG2}, for each
$y \in \mathcal X_1$ there is a nonzero $\alpha_{y} \in
\ann_{F \s A_1}(y)$, the annihilator of $y$ in 
$F \s A_1$, such that $\psi_1$ attains a unique
minimum on the support $\supp(\alpha_{y})$ of $\alpha_y$ in
$A_1$. Let $\mathcal X_2$ be a finite
$F \s A$-generating set for $M_2$.  Then 
\[ \mathcal X := \{ y \otimes z \mid y \in \mathcal X_1, z \in \mathcal X_2 \} \]
generates $M$ as an $F \s (A_1 \oplus A_2)$-module.
Denoting the image of $\alpha_y \in F \s A_1$ in
\[ F \s (A_1 \oplus A_2) = 
F \s A_1 \otimes_F F \s A_2 \]  
by $\alpha_y'$, we have
\[(y \otimes z)\alpha_y' = y\alpha_y \otimes z = 0
\otimes z = 0. \]
Furthermore, $\psi = \psi_1p_1 + \psi_2p_2$ has a
unique minimum on $\supp(\alpha_y')$. 
But then $\psi \not \in \Delta(M)$ by
 \cite[Proposition 3.1(v)]{BG2}.
This contradiction shows that
$\psi_1 \in \Delta(M_1)$.
Similarly it can be shown that
$\psi_2 \in \Delta(M_2)$. 
We have thus shown that (\ref{tsr_prod_delta}) holds.
Applying \cite[Theorem 4.4]{BG2} 
we obtain (\ref{tsr_GK_min}).
\end{proof}

\section{Strongly holonomic modules}
\label{s_h_m}
We shall now 
develop a proof of 
Theorem B.  
We shall first
prove some lemmas 
that are used in the 
proof. 
\subsection{Definitions and 
basic properties}
\begin{definition}[Definition 4.2 of \cite{BG5}]
\label{def_shol_module}
Let $M$ be a finitely generated $F \s A$-module where
$F \s A$ has center exactly $F$. Then $M$ is
\emph{strongly holonomic} if \[ \gk(M) = \frac{1}{2}\rk(A) \]
and for each commutative subalgebra $F \s C$, where $C \le A$, $M$
is torsion-free as $F \s C$-module.
\end{definition}

\begin{definition}
A nonzero $F \s A$-module $N$ is \emph{critical} if $N/L$ has GK dimension strictly smaller than that of $N$
for each $0 <  L < N$.

\end{definition}
The following proposition was
first shown in \cite{BG2}.
\begin{proposition}
\label{ecm}
Let $M$ be a finitely generated nonzero
$F \s A$-module.
Then $M$ contains a critical submodule.
\end{proposition}
\begin{proof}
Amongst the nonzero submodules of $M$, choose one, $N$
say,
of minimal possible GK dimension. If $N$ is not critical
it
has a nonzero proper submodule $N_1$ with $\gk(N/N_1) =
\gk(N)$. By the minimality of $\gk(N)$, \[ \gk(N_1) =
\gk(N) . \]
Applying the same argument to $N_1$ etc., we obtain a chain
\[ N= N_0 \supset N_1 \supset N_2 \supset \cdots \] with
$\gk(N_i/N_{i + 1}) = \gk(N)$ for each $i$. By
\cite[Lemma 5.6]{MP}, this chain must terminate. But this
process halts only when it reaches a critical module.
\end{proof}

\begin{proposition}
\label{str_hol_fin_len}
Let $M$ be a strongly holonomic $F \s A$-module, where $F \s A$ has center
$F$. Then $M$ is cyclic and has finite length.
Moreover each nonzero submodule of $M$ is also strongly holonomic.
\end{proposition}
\begin{proof}
We claim that if an algebra $F \s A$ satisfies the conditions of Proposition \ref{str_hol_fin_len}, then
\[ \gk(V) \ge \frac{1}{2}\rk(A) \]
for each nonzero $F \s A$-module $V$.
Indeed, let $V'$ be a nonzero
$F \s A$-module with
$\gk(V') < \frac{1}{2}\rk(A)$.
Then by \cite[Theorem 3]{Br2},
there is a subgroup $C \le  A$ with
$\rk(C) > \frac{1}{2}\rk(A)$ such that $F \s C$ is commutative.
By Definition \ref{def_shol_module}, $M$ must be torsion-free over $F \s C$.
Hence by Definition \ref{def_dim},
$\gk(M) > \frac{1}{2}\rk(A)$. But this is contrary
to the hypothesis in the proposition.
Hence each nonzero subfactor of $M$
has the same GK dimension as $M$.
It now follows from \cite[Lemma 5.6]{MP}  that a strictly descending sequence
of submodules of $M$ halts after a finite number of steps.
Hence $M$ has finite length.
We also note that by Proposition \ref{el_prop}(vii), $F \s A$ is simple. 
It follows from \cite[Corollary 1.5]{Ba2} that $M$ is cyclic.
\end{proof}
\subsection{Carrier space subgroups}

We recall that
for a finitely generated 
$D \s A$-module $M$, $\Delta(M)$ is
a finite union of convex 
polyhedra. 
A $D \s A$-module is 
called \emph{pure} when each nonzero 
submodule of $M$ has GK dimension equal to that of $M$. 
It is not difficult to see, noting 
Proposition \ref{dim_is_exact}, that a 
critical module is pure.  
It was shown in 
\cite{Wa2} that if $M$ 
is \emph{pure} then 
$\Delta(M)$ is a (finite) union of 
convex polyhedra each having dimension 
equal to the GK dimension of $M$. 
A subspace $\mathcal V$ of $A^*$ is \emph{rationally defined} if it can
be generated by rational linear combinations of the elements of the chosen dual basis of $A^*$.   
A rational subspace $\mathcal V$ of $A^*$ is uniquely expressed as $\mathcal V = \ann(B)$ for a 
subgroup $B$ of $A$ with $A/B$ torsion-free.  
  
\begin{definition}
\label{def_c_spaces}
Let $M$ be a finitely generated critical $D \s A$-module with $\gk(M) = m$.  
Associated with the rationally defined polyhedron $\Delta(M)$ there is a finite family of $m$-dimensional rationally defined subspaces of $A^*$ which
occur as linear spans of the 
convex polyhedra 
constituting $\Delta(M)$.
These subspaces are called the \emph{carrier spaces} of $\Delta(M)$.  
\end{definition}
\begin{definition}
\label{c_space_subgrps}
A subgroup of $A$ of the form $\ann(\mathcal V)$, where $\mathcal V$ is a carrier space of $\Delta(M)$ and $M$ a finitely generated critical
$D \s A$-module, is called a \emph{carrier space subgroup} of $\Delta(M)$.
\end{definition}
We note that for a carrier space subgroup $C$ of 
$\Delta(M)$, \[ \rk(C)  = \rk(A) - \gk(M). \]  
If $C$ is a carrier space subgroup of $\Delta(M)$ 
then $M$ cannot be finitely generated as $F \s C$-module by \cite[Proposition 3.8]{BG2}.  
The following important property of carrier space subgroups was shown  in 
\cite{Br2} on which the proof of [\cite{Br2}, Theorem A] was based.

\begin{lemma}[Proposition 4.1(2) of \cite{BG5}]
\label{css_vir_ab}
Let $M$ be a critical finitely generated $F \s A$-module and $\mathcal V$ be
a
carrier space of $\Delta(M)$. Then $B := \ann (\mathcal V)$
contains
a subgroup $B_1$ of finite index such that $F \s B_1$ is
commutative.
\end{lemma}

We recall that a subgroup $B \le A$ is \emph{isolated} in $A$ if $A/B$ is torsion-free.
\begin{definition}[\cite{BG3}]
\label{def_ng_pt}
Let $M$ be a finitely generated $F \s A$-module. 
Let $\mathcal W$ be a rational subspace of $A^*$ and $B$ be the isolated subgroup of $A$ such that 
$\mathcal W = \ann (B)$. A point $\phi \in \mathcal W$ is 
said to be \emph{nongeneric} for $\mathcal W$ and $M$ if $TC_\phi(M)$ is not
$F \s C$-torsion for some infinite cyclic subgroup $C \le B$.   
\end{definition}   

The following fact was first shown in \cite[Lemma 4.5]{BG5}.
The proof was based on a geometric
characterization of nongeneric points in $\Delta(M)$.
\begin{lemma}
\label{ker_ng_pt}
Let $M$ be a critical strongly holonomic $F \s A$-module, where $\rk(A) > 2$.
For each carrier space subgroup $U$ of $\Delta(M)$,
there is a subgroup $W$ of $A$ with rank equal to
that of $U$ such that $F \s W$ is commutative and
\[ 0 <   \rk(U \cap W) < \rk(U). \]
\end{lemma}

\begin{proof}
By Lemma \ref{css_vir_ab}, $U$ has a subgroup $U'$ of finite index such that $F \s U'$ is commutative.
By Definition \ref{def_shol_module}, $M$ is torsion-free over $F \s U'$.
We claim that $M$ is $F \s U$-torsion-free.
Indeed, if the $F \s U$-torsion submodule $t_U(M)$ of $M$ is nonzero, we may pick a finitely generated nonzero $F \s U$-submodule $N$ of $t_U(M)$. 

Clearly $N$ is 
$F \s U$-torsion and $F \s U'$-torsion-free. 
In view of \cite[Proposition 2.6]{BG2}, 
$\gk(N) < \rk(U)$. Moreover by Definition \ref{def_dim}, $\gk(N) \ge \rk(U')$ since $N$ is torsion-free as 
$F \s U'$-module. We thus have a contradiction and so $M$ must be $F \s U$-torsion-free. 

By \cite[Corollary 3.7]{BG3}, $\mathcal V := \ann(U)$ contains a nonzero point $\phi$ which
is nongeneric for $\mathcal V$ and $M$.
By \cite[Lemma 3.1]{BG3}, $U$ has an infinite cyclic subgroup $C$ such that
$\phi_C \in \Delta(M \otimes_{F \s A} (F \s A)S^{-1})$, where $\phi_C$ is the character of $(A/C)^*$ induced by $\phi$ and $S = F \s C \setminus \{0\}$. 
Note that $M_C := M \otimes_{F \s A} (F \s A)S^{-1}$ is
an  $(F \s A)S^{-1}$-module and the latter a crossed product $D_C \s A/C$,
where $D_C$ denotes the quotient division ring of $F \s C$. By \cite[Lemma 4.5(2)]{BG3}, $M_c$ is critical.

Now $\phi_C$ lies in a (rationally defined) carrier space $\mathcal V_C := \ann (V/C)$ for some $V < A$. Set $K = \ker \phi$. 
Then $K/C = \ker \phi_C \ge V/C$ 
and so $V \le K$. 
It was shown in \cite[Section 2]{Br2} that 
$D_C \s V/C$ has a nonzero module that is finite dimensional as a $D_c$-space.

Note that \[ \gk(M_C) = \gk(M) - \rk(C) \]
in view of Definition \ref{def_dim}. Since $\dim \mathcal V_C = \gk(M_C)$, hence 
\[ \rk(V/C) = \rk(A/C)  - \gk(M_C) = \rk(A) - \gk(M) = m. \]
By \cite[Corollary 3.3]{AG}, $V$ contains a rank $m$ subgroup $W$ with $F \s W$
commutative. Moreover $W$ is constructed in \cite{AG} so that $W \cap C = 1$, whence 
\[ \rk(U \cap W) < \rk(U). \]
As $\phi$ is nonzero, $\rk(K) \le 2m - 1$ and since $U, W \le K$, hence $\rk(U  \cap W) \ge 1$.
\end{proof}

The next lemma is 
a generalization of [BG5, Lemma 4.4].
\begin{lemma}
\label{AG_reslt}
Suppose that $F \s A$ has a
finitely generated module
$M$ and $A$ has a subgroup $C$ with $A/C$ torsion-free,
$\rk(C) = \gk(M)$,
and $F \s C$ commutative. Suppose moreover that $M$
is not $F \s C$-torsion.
Then $C$ has a virtual complement $E$ in $A$
such that $F \s E$ is commutative. 
In fact given $\mathbb Z$-bases 
$\{ x_1,\cdots, x_r \}$ and $\{ x_1,\cdots, x_r, 
x_{r + 1}, \cdots, x_n \}$ for 
$C$ and $A$ respectively there exist monomials $\mu_j$, where $j = 
 r + 1, \cdots n $, 
in $F \s C$, 
and an integer $s > 0$ such that the monomials $\mu_j \bar {x}_j^s$ commute in $F \s A$.
\end{lemma}
\begin{proof}
Let $\bar {x_i}\bar{x_j} = q_{ij}\bar{x_j}\bar{x_i}$, where $i,j = 1, \cdots, n$ and $q_{ij} \in F^*$.   
We set $S = F \s C\setminus\{0\}$ and denote the
quotient field $(F \s C)S^{-1}$ by $F_S$.
Then $(F \s A)S^{-1}$ is a crossed product
\[ R  = F_S \s \langle x_{r + 1}, \cdots, x_{n} \rangle. \]
The corresponding module of fractions
$MS^{-1}$ is nonzero as $M$
is not $S$-torsion by the hypothesis.
By the hypothesis, 
$\gk(M)  = \rk(C)$ and so in view of
\cite[Lemma 2.3]{BG2}, 
$MS^{-1}$ is finite dimensional as an $F_S$-space.
It is shown in \cite[section 3]{AG} that if
$R$ has a module that is one dimensional
over $F_S$
then there exist monomials
$\mu_i \in F[x_1^{\pm 1}, x_2^{\pm 1}, \cdots, x_r^{\pm 1}]$
such that the monomials $\mu_ix_i$, where $r + 1 \le i \le n$,
mutually commute. Thus we may take
$E = \langle \mu_{r + 1}x_{r + 1}, \cdots, \mu_nx_n \rangle$
in this case.
But as observed in the remark following \cite[Corollary 3.3]{AG},
 the
$s$-fold exterior power $\wedge^s(MS^{-1}_{F_S})$, where
$s = \dim_{F_S} MS^{-1}$, is a
one dimensional module over 
$R' : = F_S \s^s \langle x_{r + 1}, \cdots, x_{n} \rangle$
with the $2$-cocycle being the $s$-th power of the 2-cocycle of
$R$. Thus for $R'$,
\begin{equation} 
\label{s_th_power_of_cocyle}
q'_{ij} = 
\begin{cases} 
q_{ij} &  \forall  i  \in \{ 1, \cdots,  r\}, j \in \{ 1, \cdots, n \} \\ q_{ij}^s  &  \forall i, j \in \{ r  + 1, \dots, \ n \}  
\end{cases}
\end{equation}
By \cite[Proposition 3.2]{AG}, the monomials $\mu_jx_j$ commute in $R'$, that is,
\[ 1 = 
[\overline{\mu_k} \overline {x_k},\overline{\mu_l} 
\overline{x_l}] = [\overline {\mu_k}, \overline{x_l}]
[\overline {x_k}, \overline{\mu_l}][\overline{x_k}, 
\overline{x_l} ] \ \ \ \forall k, l \in \{ r + 1, \cdots, n \}. \]   
In view of (\ref{s_th_power_of_cocyle}), 
$[\overline {\mu_k}, \overline{x_l}]$ and $[\overline {x_k}, \overline{\mu_l}]$  
are the same in $R$ and $R'$ but $[\overline{x_k}, 
\overline{x_l} ]$ is its $s$-th power in $R'$. 
It easily follows from this that
the elements $\{\mu_jx_j^s\}_{j = r+1}^n$ commute in $R$. 
\end{proof}
The final lemma of this section is 
somewhat technical and is used in the proof of Theorem B.
\begin{lemma}
\label{four_subgroups}
Suppose that $F \s A$ 
has a strongly
holonomic module. If 
$\rk(A) = 2m$, where $m > 1$, then $A$ has 
nontrivial subgroups 
$B_i$, where $i = 1, \cdots, 4$, such that 
$F \s B_i$ is commutative and 
which satisfy the following conditions: 
\begin{align*}
 [\overline {B_1} , \overline {B_2}] =
[\overline {B_2}, \overline {B_3}] = 
[\overline {B_3}, \overline {B_4}] &= 
1 , \\   
B_1 \cap B_2 = 
B_3 \cap B_4 = 
B_1B_2 \cap B_3B_4 &= 
 1 , \\
\rk(B_1) + \rk(B_2) = \rk(B_2) + \rk(B_3) = 
\rk(B_3) + \rk(B_4) &= m.  
\end{align*}
\end{lemma}
\begin{proof}
By Proposition \ref{str_hol_fin_len}, $M$ contains a simple 
submodule which is also strongly holonomic. Hence we may assume $M$ is 
simple.

Let $U$ be a carrier space subgroup of 
$\Delta(M)$. As shown in the first paragraph of the proof of Lemma \ref{ker_ng_pt}, $M$ is torsion-free as 
$F \s U$-module. Moreover as noted above $\rk(U) = m = \gk(M)$.
Let $V$ be a (virtual) complement to $U$ in $A$ as given by Lemma \ref{AG_reslt} such that $F \s V$ is commutative. By Lemma \ref{css_vir_ab} there is a finite 
index subgroup
$U_0 \le U$ so that 
$F \s U_0$ is
commutative. But then 
$A_0 : =U_0V$ has  
finite index in $A$. 
In particular, $M$ may be regarded as a finitely generated $F \s A_0$-module $M_0$. By \cite[Lemma 2.7]{BG2}, $\gk(M_0) = \gk(M)$ and it follows that $M_0$ 
is a strongly holonomic $F \s A_0$-module.
For this reason we will assume that 
$A = UV$ with $F \s U$ and $F \s V$  commutative.
   
By Lemma \ref{ker_ng_pt}, there is also a subgroup $W$ with $\rk(W) = \rk(U) = m$
which intersects nontrivially with $U$.
Set $B_2 : = U \cap W$ and pick a subgroup $B_1$ in  $U$ maximal with respect to $B_1 \cap B_2 = 1$. 
Let $p_V :  A = U \oplus V \rightarrow V$ be the projection and $p'_V$ its
restriction to $W$. Then $\ker p'_V = B_2$ and so
\[ \rk(p'_V(W)) + \rk(B_2) = \rk(W) = m. \]
Set $B_3:= p'_V(W)$ and let $B_4 \le V$ be a subgroup maximal
with respect to
\[ B_3 \cap B_4 = 1. \]  
As $B_1, B_2 \le U$ and $F \s U$ is commutative, hence
$[\overline{B_1}, \overline{B_2}] = 1$ and similarly we can show that 
$[\overline{B_3}, \overline{B_4}] = 1$.
We claim that $[\overline{B_2}, \overline{B_3} ] = 1$.
Indeed, let $u_2 \in B_2 = U \cap W$ and $v_3 \in B_3$.
As $B_3 = p'_V(W)$, hence $uv_3 \in W$ for some
$u \in U$. Since $F \s W$ is commutative, hence
\[ 1 = [\bar u \bar {v_3}, \bar {u_2} ] =
[\bar u , \bar u_2][\bar {v_3}  , \bar u_2 ]. \] 
Moreover, as $F \s U$ is commutative, hence
$[\bar u , \bar u_2] = 1$ and
so $[\bar {v_3}  , \bar u_2 ] = 1$.
\end{proof}

We shall now give a proof of Theorem B.

\begin{theorem_T_4.2_BG5}
Suppose that an algebra $F \s A$ with center $F$ and $\rk(A) = 2m$ has a strongly holonomic module.
Then for a finite index subgroup $A'$ in $A$,
\[ F \s A' =  F \s B_1 \otimes_F  \cdots \otimes_F F \s B_m, \]
where each $B_i \cong \mathbb Z \oplus \mathbb Z$ and
$m = \frac{1}{2}\rk(A)$.
\end{theorem_T_4.2_BG5}

\begin{proof}
We shall use the notation 
\[ F \s A \stackrel{\mbox{vir}}{=} 
F \s A_1 \otimes_F \cdots \otimes_F F \s A_k \] 
to express that $A$ has a subgroup $A'$ of finite index 
such that 
\[ F \s A' = F \s A_1 \otimes_F \cdots  \otimes_F F\s A_k. \]    
We shall prove the theorem using induction. As there is
nothing to be proved for $\rk(A) = 2$,
we  assume that
$\rk(A) = 2m$, where $m > 1$. We also assume
that the theorem
holds for all smaller values of $m$ and for all
fields $F$. 

Let $B_j$, where $j = 1, \cdots, 4$, be as in
Lemma \ref{four_subgroups} and set $B: = \prod_{i = 1}^4 B_i$.
By the same lemma, $F \s B_2B_3$ is commutative and \[ \rk(B_2B_3) = m = \gk(M). \]
We fix bases in the subgroups $B_j$, where $j = 1, \cdots, 4$, as follows:
\begin{align*}
B_1 &:= \langle u_{k + 1}, \cdots, u_m \rangle,\\
B_2 &:=  \langle u_1, \cdots, u_k \rangle, \\
B_3 &:=  \langle w_{k +1}, \cdots, w_m  \rangle, \\
B_4 &:=  \langle  w_1, \cdots, w_k \rangle.
\end{align*}
By Lemma \ref{AG_reslt}
(with $C = B_2B_3$),
there are monomials $\mu_j \in F \s B_2$
and $\nu_j \in F \s B_3$, where $j = 1, \cdots m$,
such that the monomials in 
\begin{equation}
\label{ind_comt_elmts}
 \{\mu_i\nu_i \bar{w}_i^s \}_{i = 1, \cdots, k} \  \cup \  \{ \mu_j\nu_j\bar{u}_j^s \}_{j = k + 1, \cdots, m} 
\end{equation}
commute mutually for some $s > 0$.
Set  
\begin{equation}
\label{ch_var}
\bar{w}_i'  = \nu_i \bar {w}_i^s,  \ \ \ 
\bar{u}_j'  = \mu_j \bar{u}_j^s  \ \ \
\forall i \in \{1, \cdots, k\},  j \in \{k + 1, \cdots, m\}.
\end{equation}  
As just noted in (\ref{ind_comt_elmts}), 
\begin{align*}
1 &= [\mu_i\nu_i\overline{w_i}^s,
\mu_j{\nu_j}\overline{u_j}^s] \\ 
&= 
[{\mu_i}, {\mu_j}{\nu_j}\overline{u_j}^s][{\nu_i}\overline{w_i}^s , \mu_j{\nu_j}\overline{u_j}^s] \\ 
&= [ {\mu_i},  {\mu_j}{\nu_j} ][ {\mu_i},  \overline{u_j}^s][{\nu_i} \overline{w_i}^s , {\mu_j}   {\nu_j}  \overline{u_j}^s].
\end{align*}
But $[ {\mu_i},  {\mu_j}  {\nu_j}] = 1$ since $F \s B_2B_3$ is commutative and 
$[ {\mu_i},  \overline{u_j}^s] = 1$ because by Lemma
\ref{four_subgroups}, $[\overline{B_1}, \overline{B_2}] = 1$. 
We thus have
\begin{align*}
1 &= [{\nu_i} \ \overline{w_i}^s , {\mu_j}{\nu_j}  \overline{u_j}^s]  \\ 
&= [{\nu_i} \ \overline{w_i}^s , {\mu_j}    \overline{u_j}^s][{\nu_i}  \overline{w_i}^s, {\nu_j}] \\
&= 
[{\nu_i}  \overline{w_i}^s , {\mu_j}    \overline{u_j}^s]
[{\nu_i}, {\nu_j}][\overline{w_i}^s, {\nu_j}]. 
\end{align*}
By Lemma 
\ref{four_subgroups}, $F \s B_3$ is commutative and $[\overline{B_3}, \overline{B_4}] = 1$. 
It follows that 
\[ [{\nu_i}, {\nu_j}] = [\overline{w_i}^s, {\nu_j}]  = 1. \]  Hence in view of (\ref{ch_var}),  
\begin{equation}
[\overline{w_i'}, \overline{u_j'}] = 1.
\end{equation}
Setting $B_1' : = \langle u_{k + 1}', \cdots, u_m' \rangle$ and
$B_4' : = \langle w_1', \cdots, W_k' \rangle $,
we have
\[ [\overline{B_1'}, \overline{B_4'}] = \langle 1 \rangle, \]
which in view of Lemma \ref{four_subgroups} gives
\begin{equation}
\label{comtr_rl_2}
 [\overline{B_1'B_3}, \overline{B_2B_4'}] = \langle 1 \rangle. 
\end{equation}
Hence
\begin{equation}
\label{comtr_rl_2'}
F \s B = F \s B_1'B_3 \otimes_F F \s B_2B_4'.
\end{equation}
By the hypothesis in the theorem, $F \s A$
has center exactly $F$,
hence in view of (\ref{comtr_rl_2}),
$F \s B_1'B_2B_3$ has center exactly
$F \s B_2$.
Moreover, $\mathcal C : = F \s B_1'B_2$ is commutative and $M$ is thus torsion-free over $\mathcal C$. Hence for a finitely generated critical $F \s B_1'B_2B_3$-submodule $N$ of $M$, $\gk(N) = m$.
Localizing  $F \s B_1'B_2B_3$ at
$F \s B_2\setminus \{0\}$ we obtain $F' \s B_1'B_3$, where $F'$ is the quotient field
of the integral domain $F \s B_2$.

We claim that
$M': = M(F \s B_2)^{-1}$
is a strongly holonomic
$F' \s B_1'B_3$-module.
Indeed, in view of
\cite[Lemma 4.5(2)]{BG3}, 
\[ \gk(M') = \gk(M) - k = m - k =
\frac{1}{2}\rk(B_1'B_3), \] and $M'$ is
$F' \s C$-torsion-free, whenever
$F' \s C$ is commutative (see \cite[Lemma 4.3]{BG5}).
We note that the 2-cocycle of $F' \s B_1'B_3$ is the restriction of 
the 2-cocycle of $F \s B_1'B_2B_3$ to $B_1'B_3$. 
Then the induction hypothesis yields:
\begin{equation}
\label{comtr_rl_3}
F \s B_1'B_3 \stackrel{\mbox{vir}}{=} F \s C_1 \otimes_F F \s C_2 \otimes_F \cdots \otimes
F \s C_{m - k}.
\end{equation}
By parallel
reasoning applied to $F \s B_2B_3B_4'$
(which has center $F \s B_3$),
we obtain:
\begin{equation}
\label{comtr_rl_4}
F \s B_2B_4' \stackrel{\mbox{vir}}{=}
F \s E_{1} \otimes_F F \s E_2 \otimes_F \cdots \otimes_F
F \s E_{k}. 
\end{equation}
Combining (\ref{comtr_rl_2'}), (\ref{comtr_rl_3}) and
(\ref{comtr_rl_4}) shows the assertion in the theorem.
\end{proof}

\section{Nonholonomic simple modules}

We now consider the problem of the GK dimensions of simple modules over the algebras $F \s A$.
In particular, we wish to show there can be simple $F \s A$-modules with distinct GK dimensions. 
We shall accomplish this by embedding $F \s A$ in a principal ideal domain (PID).      
Given an algebra $F \s A$, let $B$ be a subgroup of $A$ with $A/B$ infinite cyclic.
The localization $F \s A(F \s B)^{-1}$ is a crossed product $D \s A/B$,  where $D$ denotes the 
quotient division ring $F \s B(F \s B)^{-1}$.
Moreover, if $A/B = \langle uB \rangle$ then $D \s A/B$ is a skew-Laurent extension 
$D[\bar u^{\pm 1}, \sigma]$, where $\sigma(d) = \bar u d \bar u^{-1}$ for all $d \in D$.

\begin{definition}
Let $R$ be a PID. An element
$r \in R$ is \emph{irreducible} when $r = st$, where $s, t \in R$, implies that
either $s$ or $t$ is a unit in $R$. 
\end{definition}
Theorem C in section \ref{introd} immediately follows from the next proposition. 
\begin{proposition}
\label{n_hol_sim_mod_c}
Suppose that $F \s A$
has center exactly $F$
and $A$ has
a subgroup $B$ with $F \s B$ commutative and 
with $A/B$ infinite cyclic.
Then $F \s A$ has a simple module $S_1$ with $\gk(S_1) = 1$.
Furthermore, let $A/B =
\langle uB \rangle$  and $R$ be the right Ore localization
$R := F \s A(F \s B)^{-1}$.
Let $r$ be an irreducible element in the PID $R$.
If $\mathcal J : = F \s A \cap rR$ contains a nonzero element $\gamma$
such that in the (unique) expression
\begin{equation}
\label{unit_poly}
\gamma = \sum_{i = s}^t  \beta_i\bar u^i, 
\end{equation}
where  $s, t \in \mathbb Z$, $\beta_i \in F\s B$ and $\beta_s$ and $\beta_t$ are units in $F \s B$, then $S_2 : = F\s A/ \mathcal J$
is a simple $F \s B$-torsion-free module with
$\gk(S_2) = n - 1$.
\end{proposition}

We shall first show the following lemma which is used in the proof of Proposition \ref{n_hol_sim_mod_c}.
\begin{lemma}
\label{f.g_ovB_t.f._ovB}
Let $F \s A$ and $F \s B$
be as in Proposition \ref{n_hol_sim_mod_c}.
Then any nonzero finitely generated $F \s A$-module $M$ which is finitely generated as a $F \s B$-module is $F \s B$-torsion-free.
\end{lemma}
\begin{proof}
Let $\rk(A) = t$.
Suppose to the contrary that the $F \s  B$-torsion submodule $T$ of $M$ is nonzero. 
Since $F \s B\setminus \{0\}$ is a right Ore subset in $F \s A$, therefore $T$ is an $F \s A$-submodule of $M$.
Since  $F \s B$ is noetherian, the hypothesis in the lemma that $M$ is 
finitely generated as $F \s B$-module implies that $T$ is finitely generated as 
$F \s B$-module. 

It follows by \cite[Lemma 2.7]{BG2} that the GK dimension of
$T$ as $F \s A$-module equals that as $F \s B$-module. 
But $T$ is by definition $F \s B$-torsion and so \[ \gk(T) < \rk(B) = t - 1 \] 
by Definition \ref{def_dim} and \cite[Proposition 2.6]{BG2}.

We shall assume for clarity that $\gk(T) = t - 2$ for our reasoning below is equally valid for all possibilities of $ \gk(T) < t - 1$.
By Definition \ref{def_dim}, there is a subgroup $C \le B$ with rank $t - 2$ so that $T$ is not $F \s C$-torsion and in view of \cite[Lemma 2.6]{BG2},
$C$ may be picked so that $B/C$ is infinite cyclic.
We pick a basis $\{v_1, v_2, \cdots, v_{t - 2} \}$ in $C$. 
Since $B/C \cong \mathbb Z$, this can be extended to a basis $\{v_1, v_2, \cdots, v_{t - 2}, v_{t -1} \}$ of $B$.
By Lemma \ref{AG_reslt}, there are monomials $\mu_{t - 1}, \mu_t \in F \s C$ and an integer $s > 0$ 
so that \[  [\mu_{t - 1}\bar {v}_{t - 1}^s, \mu_t\bar {u}^s] = 1.  \]  
Hence in view of (\ref{c_m_r_1}),
\begin{equation} \label{com_0}
[\mu_{t - 1}\bar {v}_{t - 1}^s, \mu_t][\mu_{t - 1}\bar {v}_{t - 1}^s, \bar u^s] = 1.
\end{equation}
By the hypothesis in the lemma $F \s B$ is commutative and hence $[\mu_{t - 1}\bar {v}_{t - 1}^s, \mu_t] = 1$.
Thus by (\ref{com_0}) we get noting (\ref{c_m_r_1}) -- (\ref{c_m_r_2}) that  
\begin{equation}\label{com_1}
1 = [\mu_{t  - 1}\bar {v}_{t - 1}^s, \bar u^s] = [\mu_{t  - 1}^s\bar {v}_{t - 1}^{s^2}, \bar u] .
\end{equation} 
But (\ref{com_1}) implies that the nontrivial monomial $\mu_{t  - 1}^s\bar {v}_{t - 1}^{s^2}$
is central in $F \s A$ contrary to the assumption in the lemma that $F \s A$ has center $F$.

\end{proof}

\begin{proof}[Proof of Proposition \ref{n_hol_sim_mod_c}]
We shall first show that $S_2$ is simple with $\gk(S_2) = n - 1$.
As noted above, $R$ is a PID.
Thus if $r$ is an irreducible in $R$
then $rR$ is a maximal right ideal in
$R$.

By \cite[Lemma 3.3]{BVO},
$\mathcal J = F \s A \cap rR$
is a maximal right ideal of
$F \s A$
if and only if for each
$\beta \in F \s B \setminus \{0\}$,
\begin{equation}
\label{BVO_reslt}
F \s A = \beta F \s A + \mathcal J.
\end{equation}
We shall show that $S_2 = F \s A/\mathcal J$ is simple by showing that the equality (\ref{BVO_reslt})
is satisfied.
Indeed, by the hypothesis in the theorem, $\mathcal J$
contains a nonzero element $\gamma$ of the form
(\ref{unit_poly}) and so by \cite[Proposition 2.1]{A3},
$S_2 = F \s A/ \mathcal J$ is a
finitely generated
$F \s B$-module. Setting
$\mathcal J_\beta: = \beta (F \s A) + \mathcal J$,
for $\beta \in F \s B \setminus \{0\}$, it follows that
$M_\beta: = F \s A/ \mathcal J_\beta$
is a finitely generated $F \s B$-module.
We note that if $M_\beta \ne 0$,
it has a nonzero element $m$, namely the coset
$1 + \mathcal J_\beta$, such that $m\beta = 0$
for the nonzero $\beta \in F \s B$.
But this is a contradiction in view of Lemma \ref{f.g_ovB_t.f._ovB}.
Hence $M_{\beta} = 0$ and it follows that (\ref{BVO_reslt}) is satisfied and thus
$S_2$ is simple. As already noted above $S_2$ is finitely generated as an
$F \s B$-module and so by Lemma \ref{f.g_ovB_t.f._ovB} is $F\s B$-torsion-free. 
Hence by Definition \ref{def_dim}, $\gk(S_2) \ge n - 1$. But $\gk(S_2) = n$ is impossible for
it implies, noting Definition \ref{def_dim}, that $S_2 \cong F \s A$ and
it is easily seen that $F \s A$ is not a simple $F \s A$-module.

It remains to show that $F \s A$ has a simple module $S_1$ with $\gk(S_1) = 1$.
By \cite[Section 2]{Br2}, $F \s A$ has a finitely generated module $T_1$ with $\gk(T_1) = 1$.
We claim that if $N$ is a finitely generated $F \s A$-module with 
$\gk(N) = 0$ then $N = 0$. Indeed, if $N \ne 0$ then by \cite[Theorem 3]{Br2}, $A$ has a subgroup $A'$ with 
$[A : A'] < \infty$ such that $F \s A'$ is commutative.
It is easily seen, noting (\ref{c_m_r_1}) -- (\ref{c_m_r_2}), that in this case $F \s A$ has center larger than $F$ contrary to the hypothesis in Proposition \ref{n_hol_sim_mod_c}.   

By a reasoning parallel to that in the proof of Proposition \ref{str_hol_fin_len}, it follows that $T_1$ has finite length and so contains a simple submodule $S_1$ with $\gk(S_1) = 1$.
\end{proof}

\begin{example}
Let $t > 0$ be an integer and 
$K = \mathbb Q[u_1^{\pm 1}, \cdots, u_t^{\pm 1}]$ 
be the ordinary Laurent polynomial ring over $\mathbb Q$  
in the $t$ variables $u_1, \cdots, u_n$. 
Let $p_1, p_2, \cdots, p_t$ be distinct primes in 
$\mathbb Z$.  
The skew-Laurent extension 
$T = K[u^{\pm 1}, \sigma]$, where $\sigma$ is the automorphism of $K$ defined by $\sigma(u_i) = p_iu_i$, is a 
quantum Laurent polynomial 
algebra that   
satisfies the hypothesis of 
Proposition 
\ref{n_hol_sim_mod_c}. 
Furthermore, let 
$K^* := K - \{0\}$ and $R$ be the 
right Ore localization of $T$ at 
$K^*$.        
Let $r \in R$ be an 
irreducible element of the 
form 
\[ r = u^k + f_1u^{k - 1} + 
\cdots + f_{k - 1}u + g, \]  
where $k \in \mathbb Z^+, f_1, \cdots, f_{k - 1}, g \in K$.  
and $g$ is a monomial. 
Clearly, $r$ satisfies (\ref{unit_poly}).   
By Proposition 
\ref{n_hol_sim_mod_c}, $T/T \cap rR$ is a simple $T$-module which is torsion-free over $K$.    
\end{example}

This paper is 
mainly extracted from the author's thesis and I wish to express my
sincerest gratitude to Dr. J.~R.~J.~Groves for his supervision
and to the University of Melbourne for financial support.

\end{document}